\documentclass[11pt,leqno]{amsart}

\usepackage{amsmath,amsthm}
\usepackage{latexsym}
\usepackage{amssymb}
\usepackage{mathtools}
\usepackage[parfill]{parskip}
\usepackage{mathrsfs}
\usepackage{enumerate}
\usepackage{booktabs} 
\usepackage{hyperref} 
\usepackage{xcolor} 

\newtheorem{theorem}{Theorem}

\newtheorem{lemma}[theorem]{Lemma}
\newtheorem*{claim*}{Claim}

\theoremstyle{definition}
\newtheorem{definition}[theorem]{Definition}

\numberwithin{theorem}{section}

\begin{document}

\bibliographystyle{plain}

\title{Asymptotically optimal $t$-design curves on $S^3$}
\author {Ayodeji Lindblad}
\address{Department of Mathematics, Massachusetts Institute of Technology, Cambridge, Massachusetts 02139}
\email{ayodeji@mit.edu}
\maketitle

\begin{abstract}
A \emph{spherical $t$-design curve} was defined by Ehler and Gr\"{o}chenig to be a continuous, piecewise smooth, closed curve on the sphere with finitely many self-intersections whose associated line integral applied to any polynomial of degree at most $t$ evaluates to the average of this polynomial on the sphere. These authors posed the problem of proving that there exist sequences $(\gamma_t)_{t=0}^\infty$ of $t$-design curves on $S^d$ of asymptotically optimal length $\ell(\gamma_t)\asymp t^{d-1}$ as $t\to\infty$ and solved this problem for $d=2$. This work solves the problem for $d=3$ by proving that there exists a constant $\mathscr C>0$ such that for any $C\geq\mathscr C$ and $t\in\Bbb N_+$, there exists a simple $t$-design curve on $S^3$ of length $Ct^2$.
\end{abstract}

\section{Background and main results}
\label{sec:intro}

\begin{definition}[Definition 2.1 of Ehler and Gr\"{o}chenig \cite{EhlerGrochenig23}]\label{def:design}
For any $d\in\Bbb N_+:=\{1,2,...\}$ and $t\in\Bbb N:=\{0,1,...\}$, a continuous, piecewise smooth, closed curve $\gamma:[0,1]\to S^d$ with finitely many self-intersections is called a \emph{$t$-design curve} if 
\[\frac1{\ell(\gamma)}\int_\gamma f:=\frac1{\ell(\gamma)}\int_0^1f(\gamma(s))|\gamma'(s)|\,ds=\frac1{|S^d|}\int_{S^d}f\,d\sigma\]
for all $f$ in the space $P_t(S^d)$ of restrictions to $S^d$ of polynomials on $\Bbb R^{d+1}$ of degree at most $t$, where $\ell(\gamma):=\int_0^1|\gamma'(s)|\,ds$ is the length of $\gamma$ and $\sigma$ is the standard uniform measure on $S^d$.
\end{definition}

The introduction of $t$-design curves was motivated by the numerous uses curves have found in data analysis on the sphere \cite[Section 1]{EhlerGrochenig23}. These objects are a natural analogue of the well-studied \cite{BannaiBannai09} discrete objects \emph{spherical $t$-designs}---finite point sets on spheres which provide \emph{quadrature} or \emph{cubature} rules, averaging degree at most $t$ polynomials exactly---introduced by Delsarte, Goethals, and Seidel \cite{Delsarte...77}. The optimal asymptotic order of size of a $t$-design on $S^d$ as $t\to\infty$ was shown by Delsarte, Goethals, and Seidel---who noted that the vertices of a regular $(t+1)$-gon give a $t$-design on $S^1$ \cite[Example 5.1.4]{Delsarte...77}---to be $t$ when $d=1$ and was a long-standing open problem for $d>1$ until Bondarenko, Radchenko, and Viazovska showed that an asymptotic lower bound on this size proven by Delsarte, Goethals, and Seidel \cite[Definition 5.13]{Delsarte...77} was optimal up to a constant for each $d\in\{2,3,...\}$, specifically proving existence of a constant $C_d>0$ such that there exists a spherical $t$-design on $S^d$ of size $N$ for any $t\in\Bbb N_+$ and $N\geq C_dt^d$ \cite[Theorem 1]{Bondarenko...13}. Ehler and Gr\"{o}chenig established an asymptotic lower bound $\ell(\gamma)\geq\widetilde{c}_dt^{d-1}$ ($\widetilde c_d$ a constant depending on $d$) on the order of length of a spherical $t$-design curve $\gamma$ on $S^d$ for $d\in\Bbb N_+$ \cite[Theorem 1.1]{EhlerGrochenig23} and called sequences $(\gamma_t)_{t=0}^\infty$ of $t$-design curves achieving this bound up to a constant \emph{asymptotically optimal}. These authors used the results of Bondarenko, Radchenko, and Viazovska \cite[Theorem 1]{Bondarenko...13} to show existence of asymptotically optimal sequences of $t$-design curves on $S^d$ for $d=2$ \cite[Theorem 1.2]{EhlerGrochenig23} and posed the problem of proving existence of such sequences for $d>2$. Theorem \ref{thm:seq} solves this problem for $d=3$.

\begin{theorem}[Main Theorem]\label{thm:seq}
There exists a constant $\mathscr C>0$ such that for any $C\geq\mathscr C$ and $t\in\Bbb N_+$, there exists a simple $t$-design curve on $S^3$ of length $Ct^2$.
\end{theorem}

We note that applying techniques of Ehler and Gr\"ochenig \cite[Sections 4-6]{EhlerGrochenig23} to a sequence $(\gamma_t)_{t=1}^\infty$ of $t$-design curves satisfying $\ell(\gamma_t)=\mathscr Ct^2$, one may decrease the shortest known asymptotic order of length achieved by sequences of $t$-design curves on $S^d$ for $d\geq3$ from $t^{d(d-1)/2}$ to $t^{d(d-1)/2-1}$.

To show Theorem \ref{thm:seq}, we combine a construction communicated by Theorem \ref{thm:stitch} which builds a $t$-design curve on $S^3$ from a $\lfloor t/2\rfloor$-design curve on $S^2$ with the result of Ehler and Gr\"{o}chenig that there exists an asymptotically optimal sequence of $t$-design curves on $S^2$.

\begin{theorem}\label{thm:stitch}
Consider $t\in\Bbb N$ and a $\lfloor t/2\rfloor$-design curve $\alpha$ on $S^2$. For any $\varepsilon\geq0$, we may construct a $t$-design curve $\gamma_{\alpha,\varepsilon}$ on $S^3$ as in \eqref{eq:gammaformula} which may have self-intersections for $\varepsilon=0$ but which will be simple for $\varepsilon>0$ of length
\begin{equation}\label{eq:len}
\ell(\gamma_{\alpha,\varepsilon})=(t+1)\sqrt{\ell(\alpha)^2+\phi_\alpha^2}+\varepsilon
\end{equation}
for a constant $\phi_\alpha\in(-\pi,\pi]$ chosen as in \eqref{eq:phidef} which satisfies the bound \eqref{eq:genbound}. 
\end{theorem}

\begin{theorem}[Theorem 1.2 of Ehler and Gr\"{o}chenig \cite{EhlerGrochenig23}]\label{thm:EG2}
There exists a constant $\mathscr A>0$ such that for any $t\in\Bbb N_+$, there exists a $t$-design curve on $S^2$ of length at most $\mathscr At$.
\end{theorem}

Theorem \ref{thm:seq} directly follows from Theorems \ref{thm:stitch} and \ref{thm:EG2} by taking 
\[\mathscr C:=\sqrt{\mathscr A^2+4\pi^2},\]
as for any $t\in\Bbb N_+$ and $\alpha$ a $\lfloor t/2\rfloor$-design curve satisfying $\ell(\alpha)\leq\mathscr At$, we will have $\ell(\gamma_{\alpha,\varepsilon})\leq\mathscr Ct^2$ for $\varepsilon=0$ and we may then increase the length of $\gamma_{\alpha,\varepsilon}$ as desired by increasing $\varepsilon$.
We prove Theorem \ref{thm:stitch} in Section \ref{sec:const}, then discuss examples of explicit $t$-design curves arising from the construction of Theorem \ref{thm:stitch} and results arising from natural generalizations of the theorem in Section \ref{sec:further}. These results include explicit sequences of asymptotically optimal $t$-design curves on $(S^1)^d$, existence of asymptotically optimal sequences of $t$-design curves on $S^2\times(S^1)^d$ and $S^3\times(S^1)^d$, and a construction to be made fully formal in future work \cite{LindbladWIP} which builds a $t$-design curve on $S^{2n+1}$ from a $\lfloor t/2\rfloor$-design curve on $\Bbb{CP}^n$ that we plan to use to give rise to improved bounds on the asymptotically smallest sequences of $t$-design curves known to exist on $S^d$ for $d>3$.

\section{Building $t$-design curves using the Hopf map}
\label{sec:const}

We now give an informal overview of the proof of Theorem \ref{thm:stitch}, present lemmas used in the proof in Subsection \ref{sub:lem}, and provide the formal proof in Subsection \ref{sub:thm}. 
Consider the Hopf map 
\begin{equation}\label{eq:pi}
\pi:S^3\to S^2,\quad(a,b)\in\Bbb C^2\mapsto(|a|^2-|b|^2,2a\overline b)\in\Bbb R\times\Bbb C,
\end{equation}
which gives rise to a principal $S^1$-bundle with fibers
\begin{equation}
\pi^{-1}(w)=\{\omega z\:|\:z\in S^1\subset\Bbb C\}\cong S^1
\end{equation}
for $w\in S^2$ and any $\omega\in\pi^{-1}(w)$. To build the $t$-design curve $\gamma_{\alpha,\varepsilon}$ on $S^3$ discussed in Theorem \ref{thm:stitch} from a $\lfloor t/2\rfloor$-design curve $\alpha$ on $S^2$, we first use Lemma \ref{lem:lift} to lift $\alpha$ to a curve $\beta_\alpha$ on $S^3$ satisfying $\pi\circ\beta_\alpha=\alpha$ whose derivative $\beta_\alpha'(s)$ is always orthogonal to the tangent space of the fiber $\pi^{-1}(\pi(\beta_\alpha(s)))$. We then rotate $\beta_\alpha$ fiberwise such that the concatenation $\gamma_{\alpha,\varepsilon}$ of $t+1$ appropriately rotated copies of the resulting curve will be continuous, piecewise smooth, closed, and, if $\varepsilon>0$, to remove any self-intersections and lengthen the curve by $\varepsilon$. We note that the lengthening of the curve done to ensure that $\gamma_{\alpha,\varepsilon}$ will be closed results in the constant $\phi_\alpha$ in the formula \eqref{eq:len} for the length of $\gamma_{\alpha,\varepsilon}$. We may observe from our choice \eqref{eq:phidef} of $\phi_\alpha$ that $\phi_\alpha$ may be expressed as a function of the area enclosed by $\alpha$ using the Gauss-Bonnet theorem \cite{Gauss1827,Bonnet1867}.
Additionally, denoting by $G_{t+1}$ the set of generators of the cyclic group of order $t+1$ (natural numbers less than and coprime to $t+1$), it is straightforward to see that
\begin{equation}\label{eq:genbound}
|\phi_\alpha|\leq\frac{\pi}{t+1}\max_{g_1,g_2\in G_{t+1}}\min(|g_1-g_2|,t+1-|g_1-g_2|)
\end{equation}
for $t>2$, which gives rise to the bound $|\phi_\alpha|\leq\frac{2\pi}{t+1}$ when $t+1$ is prime.

The construction described above will arrange for all $s\in[0,1]$ that
\[\quad(\pi\circ\gamma_{\alpha,\varepsilon})(s)=\alpha((t+1)s-\left\lfloor (t+1)s\right\rfloor)\]
and that
\[(\gamma_{\alpha,\varepsilon}([0,1]))\cap\pi^{-1}(\pi(\gamma_{\alpha,\varepsilon}(s)))\]
is the vertices of a regular $(t+1)$-gon on $\pi^{-1}(\pi(\gamma_{\alpha,\varepsilon}(s)))\cong S^1$, which we see from Lemma \ref{lem:DSGex}---the result of an example \cite[Example 5.1.4]{Delsarte...77} of Delsarte, Goethals, and Seidel---is a $t$-design on the fiber. We show that such a curve will be a $t$-design curve on $S^3$ using a method analogous to those used first by K\"{o}nig \cite[Corollary 1]{Koning98} and Kuperberg \cite[Theorem 4.1]{Kuperberg06} to relate $t$-designs on spheres to $\lfloor t/2\rfloor$-designs on quotient complex projective spaces, later by Okuda \cite[Theorem 1.1]{Okuda15} (who was inspired by work of Cohn, Conway, Elkies, and Kumar \cite[Section 4]{Cohn...06} on the subject) to relate $t$-designs on $S^3$ to $\lfloor t/2\rfloor$-designs on $S^2$, and most recently by the present author \cite[Theorem 1.1]{Lindblad24} to relate $t$-designs on spheres to $\lfloor t/2\rfloor$-designs on quotient real, complex, quaternionic, or octonionic projective spaces or spheres. To this end, we will present Lemmas \ref{lem:int} and \ref{lem:chpoly}, which were used by Okuda in formalizing their result \cite[Theorem 1.1]{Okuda15}.

We may observe from this outline of the construction of $\gamma_{\alpha,\varepsilon}$ and the presentation 
\begin{equation*}\label{eq:explicitfiber}
\pi^{-1}(\xi,\eta)=
\begin{cases}
\Bigl\{\frac1{\sqrt2}\Bigl(\zeta\sqrt{1+\xi},\frac{\overline\eta \zeta}{\sqrt{1+\xi}}\Bigr)\:\Big|\:\zeta\in S^1\Bigr\}, & \xi\neq-1 \\
\{(0,\zeta)\:|\:\zeta\in S^1\}, & \xi=-1
\end{cases}
\end{equation*}
of preimages of $\pi$ that 
\begin{equation}\label{eq:gammaformula}
\begin{gathered}
\gamma_{\alpha,\varepsilon}(s)=\begin{cases}\frac1{\sqrt2}\left(\sqrt{1+\alpha_1(r)},\frac{\alpha_2(r)-i\alpha_3(r)}{\sqrt{1+\alpha_1(r)}}\right)e^{2\pi is+i\theta(r)}, & \alpha_1(r)\neq-1 \\
\left(0,1\right)e^{2\pi is+i\theta(r)}, & \alpha_1(r)=-1
\end{cases}%
\end{gathered}
\end{equation}
for $r:=(t+1)s-\lfloor(t+1)s\rfloor$ and a piecewise smooth function $\theta:[0,1]\to\Bbb R$ satisfying $\theta(0)-\theta(1)\in2\pi\Bbb Z$ which will be continuous at all points $r\in[0,1]$ satisfying $\alpha_1(r)\neq-1$.

\subsection{Curves and polynomials related through the Hopf map}
\label{sub:lem}

We now present lemmas used in the proof of Theorem \ref{thm:stitch}. Lemma \ref{lem:lift} discusses how to lift a curve on $S^2$ to a curve on $S^3$ whose derivative is orthogonal to the tangent space of each Hopf fiber it passes through, Lemma \ref{lem:int} discusses how an integrable function on $S^3$ can be averaged by averaging the function on $S^2$ which averages it on each Hopf fiber, Lemma \ref{lem:chpoly} relates polynomials on $S^3$ to those on $S^2$ and on Hopf fibers, and Lemma \ref{lem:DSGex} discusses how the vertices of a regular $(t+1)$-gon on a Hopf fiber give a $t$-design on the fiber.
We first show Lemma \ref{lem:lift}. To this end, observe that since the Hopf map \eqref{eq:pi} constitutes a fiber bundle, we have the decomposition
\begin{equation}\label{eq:decomp}
T_\omega S^3\cong T_\omega(\pi^{-1}(w))\oplus T_{w}S^2
\end{equation}
for $\omega\in S^3$ and $w:=\pi(\omega)\in S^2$. There is a natural isometric inclusion 
\[\iota_\omega:T_{w}S^2\hookrightarrow T_\omega S^3\]
satisfying
\begin{equation}\label{eq:sideinv}
d\pi_\omega\circ\iota_\omega=\text{id}_{T_{w}S^2},
\end{equation}
where $\text{id}_{T_{w}S^2}$ is the identity map on $T_{w}S^2$.

\begin{lemma}\label{lem:lift}
Take a continuous, piecewise smooth curve $\alpha:[0,1]\to S^2$
with finitely many self-intersections. We may construct a continuous, piecewise smooth curve $\beta_\alpha:[0,1]\to S^3$ with finitely many self-intersections satisfying 
\begin{equation}\label{eq:betaprimelemma}
\beta_\alpha'(s)=\iota_{\beta_\alpha(s)}(\alpha'(s))
\end{equation}
for all $s\in[0,1]$ at which $\alpha$ is smooth and
\begin{equation}\label{eq:pibetalemma}
\pi\circ\beta_\alpha=\alpha.
\end{equation}
\end{lemma}

\begin{proof}
Consider $\alpha$ as in the lemma and a partition $0=s_0<\cdots<s_{m}=1$
of $[0,1]$ such that $\alpha$ is smooth and has no self-intersections on each interval $I_j:=[s_{j},s_{j+1}]$. Fixing $j\in\{0,...,m-1\}$, the restriction $\alpha_j:=\alpha|_{I_j}$ is then a diffeomorphism onto its image $\alpha(I_j)$, so $\pi^{-1}(\alpha(I_j))$ is a smooth submanifold (with boundary) of $S^3$. We define a vector field $V_j$ on $\pi^{-1}(\alpha(I_j))\ni\omega$ by
\[(V_j)_\omega=\iota_\omega(\alpha'(\alpha_j^{-1}(\pi(\omega)))).\]
The method of successive approximations allows us to construct the unique smooth flow 
\[\beta_{j}:\pi^{-1}(\alpha(s_{j}))\times I_j\to S^3\]
satisfying 
\begin{equation}\label{eq:gammadef1}
\begin{gathered}
\beta_{j}(\omega,s_{j})=\omega, \\
\frac{\partial\beta_{j}}{\partial s}(\omega,s)=(V_j)_{\beta_{j}(\omega,s)}
\end{gathered}
\end{equation}
for $\omega\in\pi^{-1}(\alpha(s_{j}))$ and $s\in I_j$.
For such $\omega$, we have
\begin{equation*}
(\pi\circ\beta_{j})(\omega,s_{j})=\pi(\omega)=\alpha(s_{j})
\end{equation*}
and \eqref{eq:sideinv} shows that for $s\in I_j$ satisfying $(\pi\circ\beta_{j})(\omega,s)=\alpha(s)$, we have
\begin{equation*}
\begin{split}
\frac{\partial(\pi\circ\beta_{j})}{\partial s}(\omega,s)&=\left(d\pi_{\beta_{j}(\omega,s)}\circ\frac{\partial\beta_{j}}{\partial s}\right)(\omega,s) \\
&=d\pi_{\beta_{j}(\omega,s)}((V_j)_{\beta_{j}(\omega,s)}) \\
&=\alpha'(\alpha_j^{-1}((\pi\circ\beta_{j})(\omega,s))) \\
&=\alpha'(s).
\end{split}
\end{equation*}
So, we must have
\begin{equation}\label{eq:pigammaalpha}
(\pi\circ\beta_{j})(\omega,s)=\alpha(s)
\end{equation}
for $\omega\in\pi^{-1}(\alpha(s_{j}))$ and $s\in I_j$, showing that $\beta_{j}$ has no self-intersections since $\alpha_j$ has none. 

Fix $\omega\in\pi^{-1}(\alpha(0))$. We now consider the curve $\beta_\alpha:[0,1]\to S^3$ which arises from flowing along each $\beta_j$ successively starting at $\omega$, so we have $\beta_\alpha(0)=\omega$ and $\beta_\alpha(s)=\beta_j(\beta_\alpha(s_{j}),s)$ for $j\in\{0,...,m-1\}$ and $s\in I_j$.
By this construction, we see since $\beta_j$ is smooth and has no self-intersections for each $j$ that $\beta_\alpha$ is continuous and piecewise smooth with finitely many self-intersections. Moreover, we directly see from \eqref{eq:gammadef1} and \eqref{eq:pigammaalpha} that \eqref{eq:betaprimelemma} and \eqref{eq:pibetalemma} hold.
\end{proof}

Lemmas \ref{lem:int} and \ref{lem:chpoly} are concerned with properties of the operator $I_\pi$ which takes $f\in L^1(S^3)$ to the function
\begin{equation} \label{eq:Ich}
(I_{\pi}f)(w)=\frac1{2\pi}\int_{\pi^{-1}(w)}f\,d\sigma
\end{equation}
on $S^2$ which gives the average of $f$ on each Hopf fiber (where $\sigma$ is the standard uniform measure on $\pi^{-1}(w)\cong S^1$) and the left multiplication by a base point $\omega$ isomorphism
\begin{equation} \label{eq:zetach}
\zeta_\omega:S^1\to\pi^{-1}(w),\quad\zeta\mapsto \omega\zeta
\end{equation}
we define for $\omega\in S^3$ and $w:=\pi(\omega)\in S^2$. We do not prove these lemmas, but note that proofs can be found in work of Okuda \cite[Lemmas 4.2-4.3]{Okuda15} or of the present author \cite[Lemmas 2.2-2.3, 2.5]{Lindblad24}. 

\begin{lemma}[Lemma 4.2 of Okuda \cite{Okuda15}]\label{lem:int}
For $f\in L^1(S^3)$, we have
\begin{equation*}
\frac1{|S^2|}\int_{S^2}I_\pi f\,d\sigma=\frac1{|S^3|}\int_{S^3}f\,d\sigma.
\end{equation*}
\end{lemma}

\begin{lemma}[Lemma 4.3 of Okuda \cite{Okuda15}]\label{lem:chpoly}
We have
\begin{gather}
I_{\pi}(P_t(S^3))=P_{\left\lfloor t/2\right\rfloor}(S^2), \label{eq:chIpoly} \\
\zeta_\omega^*(P_t(S^3))=P_t(S^1)\quad\text{\it for any}\quad\omega\in S^3.\label{eq:chzetapoly}
\end{gather}
\end{lemma}

Lemma \ref{lem:DSGex} follows from \eqref{eq:chzetapoly} paired with the fact \cite[Example 5.1.4]{Delsarte...77} noted by Delsarte, Goethals, and Seidel that the vertices of a regular $(t+1)$-gon give a $t$-design on $S^1$.

\begin{lemma}[Example 5.1.4 of Delsarte, Goethals, and Seidel \cite{Delsarte...77}]\label{lem:DSGex}
For any $\omega\in S^3$ and $f\in P_t(S^3)$, we have
\[\frac1{t+1}\sum_{j=0}^tf(\omega e^{2\pi ij/(t+1)})=(I_\pi f)(\pi(\omega)).\]
\end{lemma}

\subsection{Formal treatment of the construction}
\label{sub:thm}

We now prove Theorem \ref{thm:stitch}, applying Lemma \ref{lem:lift} to assist in constructing the desired curve $\gamma_{\alpha,\varepsilon}$ and using Lemmas \ref{lem:int}, \ref{lem:chpoly}, and \ref{lem:DSGex} to show that $\gamma_{\alpha,\varepsilon}$ is a $t$-design curve.

\begin{proof}[Proof of Theorem \ref{thm:stitch}]
Consider $\alpha$ as in the theorem statement, which we reparameterize so its derivative has constant norm $|\alpha'|=\ell(\alpha)$.
Consider $\beta_\alpha$ as in Lemma \ref{lem:lift} alongside the generators $G_{t+1}$ of the cyclic group of order $t+1$.
Since $\alpha$ is a closed curve, \eqref{eq:pibetalemma} shows that $\beta_\alpha(0)$ lies in the image $\pi^{-1}(\pi(\beta_\alpha(1)))$ of the isomorphism $\zeta_{\beta_\alpha(1)}$ as in \eqref{eq:zetach}. Therefore, we may pick $g\in G_{t+1}$ minimizing $|\phi_\alpha|$ for $\phi_\alpha\in(-\pi,\pi]$ defined by
\begin{equation}\label{eq:phidef}
e^{i\phi_\alpha}=\zeta_{\beta_\alpha(1)}^{-1}(\beta_\alpha(0))e^{2\pi ig/(t+1)}\in S^1\subset\Bbb C.
\end{equation}
With $\varepsilon$ as in the theorem statement, fix 
\[\phi_\varepsilon:=\sqrt{\phi_\alpha^2+\frac{2\varepsilon\sqrt{\ell(\alpha)^2+\phi_\alpha^2}}{t+1}+\frac{\varepsilon^2}{(t+1)^2}},\]
so we have
\begin{equation}\label{eq:ddef}
\sqrt{\ell(\alpha)^2+\phi_\varepsilon^2}=\sqrt{\ell(\alpha)^2+\phi_\alpha^2}+\frac{\varepsilon}{t+1}.
\end{equation}
We may consider the partition $0=s_0<\cdots<s_{m}=1$
of $[0,1]$ arising from the union of the set of self-intersection points of $\alpha$ with $\{0,1\}$.
We define
\begin{equation*}
\begin{gathered}\label{eq:paRtition}
r_{j,\delta}:=\frac12(s_{j-1}+s_j+(s_{j}-s_{j-1})\phi_\alpha/\phi_\varepsilon)+\delta\quad\text{\it for}\quad j\in\{1,...,m-1\},\\
r_{m,\delta}:=\frac12(s_{m-1}+1+(1-s_{m-1})\phi_\alpha/\phi_\varepsilon)-(m-1)\delta
\end{gathered}
\end{equation*}
for some fixed $\delta\in[0,\Delta]$, where we have 
\[\Delta=\min\left\{s_{1}-r_{1,0},...,s_{m-1}-r_{m-1,0},\frac{r_{m,0}-s_{m-1}}{m-1}\right\}\] 
so that $r_{j,\delta}\in[s_{j-1},s_j]$ for all $j$. Setting
\[I_{+,\delta}:=\bigcup_{j=1}^{m}(s_{j-1},r_{j,\delta}),\quad I_{-,\delta}:=\bigcup_{j=1}^{m}(r_{j,\delta},s_{j}),\]
we can directly compute that
\begin{gather*}
|[0,s_j]\cap I_{+,\delta}|-|[0,s_j]\cap I_{-,\delta}|=s_j\phi_\alpha/\phi_\varepsilon+2j\delta\quad\text{\it for}\quad j\in\{1,...,m-1\},\\
|I_{+,\delta}|-|I_{-,\delta}|=\phi_\alpha/\phi_\varepsilon.
\end{gather*}
Thus, considering the unique continuous function $\theta_\delta:[0,1]\to\Bbb R$ defined by 
\begin{equation*}\label{eq:theta0}
\theta_\delta(0)=0,\quad\theta_\delta'(s)=\pm\phi_\varepsilon\quad\text{\it for}\quad s\in I_{\pm,\delta},
\end{equation*}
we have that
\begin{equation}\label{eq:theta1}
\begin{gathered}
\theta_\delta(s_j)=s_j\phi_\alpha+2j\delta\phi_\varepsilon\quad\text{\it for}\quad j\in\{0,...,m-1\},\\
\theta_\delta(1)=(|I_{+,\delta}|-|I_{-,\delta}|)\phi_\varepsilon=\phi_\alpha.
\end{gathered}
\end{equation}
Defining functions
\[q(s)=\frac{\left\lfloor(t+1)s\right\rfloor}{t+1},\quad r(s)=(t+1)(s-q(s))\]
for $s\in[0,1]$, we consider the continuous, piecewise smooth curve 
\[\widetilde\gamma_{\alpha,\delta}:=(\beta_\alpha\circ r)e^{i(\theta_\delta\circ r+2\pi gq)}:[0,1]\to S^3.\]
We may directly see that 
\[\widetilde\gamma_{\alpha,\delta}(1)=\beta_\alpha(0)=\widetilde\gamma_{\alpha,\delta}(0),\]
so $\widetilde\gamma_{\alpha,\delta}$ is a closed curve.
Observe from \eqref{eq:pibetalemma} that
\begin{equation*}\label{eq:pigammaa2}
\pi\circ\widetilde\gamma_{\alpha,\delta}=\pi\circ\beta_\alpha\circ r=\alpha\circ r,
\end{equation*}
so $\widetilde\gamma_{\alpha,\delta}$ may only have a self-intersection at a point $s\in[0,1]$ if there exists $\tilde s\in[0,1]$ such that $r(s)=r(\tilde s)$ or if $r(s)$ is a self-intersection point of $\alpha$. If $r(s)=r(\tilde s)$, we have $\tilde s=s+k/(t+1)$ for some $k\in\{-t,...,t\}$, so 
\[\widetilde\gamma_{\alpha,\delta}(\tilde s)=\beta_\alpha(r(s))e^{i(\theta_\delta(r(s))+2\pi g(q(s)+k/(t+1)))}=\widetilde\gamma_{\alpha,\delta}(s)e^{2\pi igk/(t+1)}\neq \widetilde\gamma_{\alpha,\delta}(s)\]
since $g$ is a generator of the cyclic group of $t+1$ elements and thus $gk$ will not be an integer multiple of $t+1$. Therefore, $\widetilde\gamma_{\alpha,\delta}$ may only have self-intersections in
\[r^{-1}(\{s_j\}_{j=0}^m)=\left\{s_{j,k}:=\frac{s_j+k}{t+1}\:\middle|\:j\in\{0,...,m\},\:k\in\{0,...,t\}\right\}.\]
For any $k\in\{0,...,t\}$, \eqref{eq:theta1} and \eqref{eq:phidef} show that
\begin{equation*}
\begin{gathered}
\label{eq:gammasjk}
\widetilde\gamma_{\alpha,\delta}(s_{j,k})=\beta_\alpha(s_j)e^{i(s_j\phi_\alpha+2j\delta\phi_\varepsilon+2\pi gk/(t+1))}\quad\text{\it for}\quad j\in\{0,...,m-1\},\\
\widetilde\gamma_{\alpha,\delta}(s_{m,k})=\beta_\alpha(1)e^{i(\phi_\alpha+2\pi gk/(t+1))}=\beta_\alpha(0)e^{2\pi ig(k+1)/(t+1)},
\end{gathered}
\end{equation*}
so we can see that there will only be finitely many $\delta\in[0,\Delta]$ such that $\widetilde\gamma_{\alpha,\delta}$ will have any self-intersections.
We take $\gamma_{\alpha,\varepsilon}:=\widetilde\gamma_{\alpha,0}$ when $\varepsilon=0$. Otherwise, we have $\Delta>0$, so we may pick $\delta\in[0,\Delta]$ such that $\widetilde\gamma_{\alpha,\delta}$ has no self-intersections and label $\gamma_{\alpha,\varepsilon}:=\widetilde\gamma_{\alpha,\delta}$.

Combining \eqref{eq:betaprimelemma}, the fact that $\iota_\omega$ is an isometry for all $\omega\in S^3$, and our assumption that $|\alpha'|$ is constant, we see that $|\beta_\alpha'|=|\alpha'|=\ell(\alpha)$, so the decomposition \eqref{eq:decomp} shows that
\begin{equation}\label{eq:deriv}
|\gamma_{\alpha,\varepsilon}'|^2=|(\beta_\alpha\circ r)'|^2+|(\theta_\delta\circ r)'|^2=(t+1)^2(\ell(\alpha)^2+\phi_\varepsilon^2)
\end{equation}
at all points where $\gamma_{\alpha,\varepsilon}$ is smooth. Thus, we have
\[\ell(\gamma_{\alpha,\varepsilon})=(t+1)\sqrt{\ell(\alpha)^2+\phi_\varepsilon^2}\]
and \eqref{eq:ddef} then shows that \eqref{eq:len} is satisfied.

To complete the proof, we need only show that for any $f\in P_t(S^3)$,
\begin{equation}\label{eq:feq}
\frac1{\ell(\gamma_{\alpha,\varepsilon})}\int_{\gamma_{\alpha,\varepsilon}}f=\frac1{|S^3|}\int_{S^3}f\,d\sigma.
\end{equation}
Picking such $f$, as \eqref{eq:deriv} shows that $|\gamma'_{\alpha,\varepsilon}|$ is constant and thus equals $\ell(\gamma_{\alpha,\varepsilon})$, we have
\begin{equation}\label{eq:design1}
\frac1{\ell(\gamma_{\alpha,\varepsilon})}\int_{\gamma_{\alpha,\varepsilon}}f=\frac1{\ell(\gamma_{\alpha,\varepsilon})}\int_0^1f(\gamma_{\alpha,\varepsilon}(s))\ell(\gamma_{\alpha,\varepsilon})\,ds=\int_0^1f(\gamma_{\alpha,\varepsilon}(s))\,ds.
\end{equation}
With $I_\pi$ as in \eqref{eq:Ich}, we then see from Lemma \ref{lem:DSGex} and since $g$ is a generator of the cyclic group of order $t+1$ that
\[\frac1{t+1}\sum_{k=0}^tf(\omega e^{2\pi igk/(t+1)})=(I_\pi f)(\pi(\omega)),\]
so applying a change of variables $s\mapsto\frac s{t+1}$, \eqref{eq:pibetalemma}, and that $|\alpha'|=\ell(\alpha)$, we have
\begin{equation}\label{eq:design2}
\begin{split}
\int_0^1f(\gamma_{\alpha,\varepsilon}(s))\,ds&=\sum_{k=0}^{t}\int_{k/(t+1)}^{(k+1)/(t+1)}f(\gamma_{\alpha,\varepsilon}(s))\,ds \\
&=\int_0^{1}\frac1{t+1}\sum_{k=0}^tf(\beta_\alpha(s)e^{i(\theta_\delta(s)+2\pi gk)})\,ds \\
&=\int_0^{1}(I_\pi f)((\pi\circ\beta_\alpha)(s))\,ds \\
&=\int_0^1(I_\pi f)(\alpha(s))\,ds \\
&=\frac1{\ell(\alpha)}\int_\alpha I_\pi f.
\end{split}
\end{equation}
We see from \eqref{eq:chIpoly} in Lemma \ref{lem:chpoly} that $I_\pi f\in P_{\left\lfloor t/2\right\rfloor}(S^2)$, so since $\alpha$ is a $\lfloor t/2\rfloor$-design, \eqref{eq:design1} and \eqref{eq:design2} combine to show that
\[\frac1{\ell(\gamma_{\alpha,\varepsilon})}\int_{\gamma_{\alpha,\varepsilon}}f=\frac1{|S^2|}\int_{S^2}I_\pi f\,d\sigma.\]
Lemma \ref{lem:int} then shows that \eqref{eq:feq} is satisfied.
\end{proof}

\section{Explicit $t$-design curves and generalizations}
\label{sec:further}

We now discuss examples of explicit $t$-design curves arising from the construction of Theorem \ref{thm:stitch}. To this end, observe that the curve $\alpha(s)=(0,\cos(2\pi s),\sin(2\pi s))$ which traces around the equator of $S^2$ is a 1-design curve on $S^2$. Taking $t\in\{2,3\}$ and building $\gamma_{\alpha,0}$ as in Theorem \ref{thm:stitch}, we produce a smooth, simple $t$-design curve 
\begin{equation}\label{eq:23desgs}
[0,1]\to S^3,\quad s\mapsto\frac1{\sqrt2}(\cos(2\pi s),\:\sin(2\pi s),\:\cos(2\pi ts)\:,-\sin(2\pi ts))
\end{equation}
on $S^3$ which we may observe has length $\pi\sqrt{2t^2+2}$. This curve for $t=3$ is pictured in Figure \ref{fig:S3ex}.

\begin{figure}
\begin{center}
\includegraphics[width=.55\textwidth]
{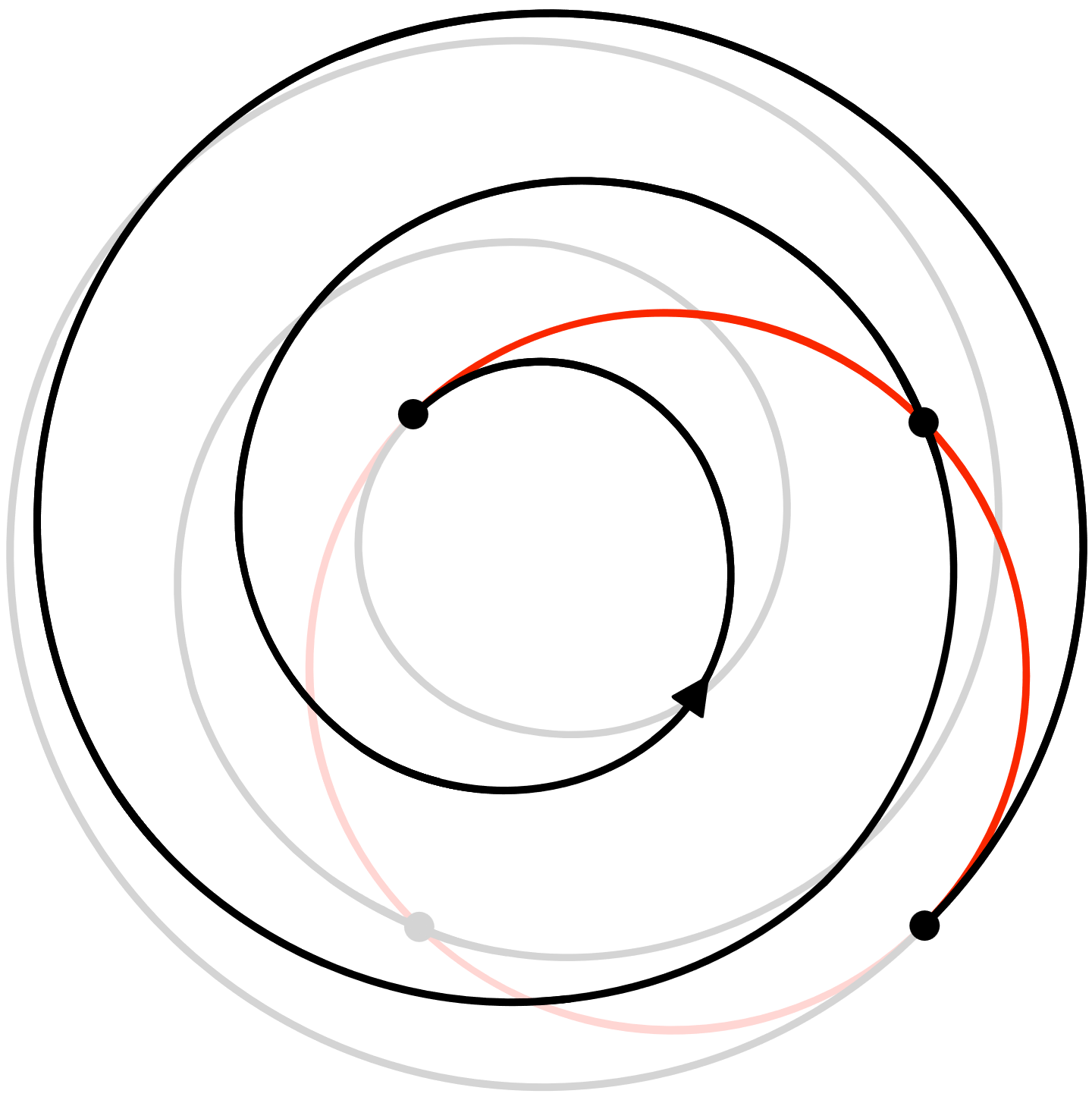}
\caption{\label{fig:S3ex} The 3-design curve $\gamma_{\alpha,0}$ on $S^3\cong\Bbb R^3\cup\{\infty\}$ as in \eqref{eq:23desgs} for $\alpha$ the curve which traces around the equator in $S^2$. $\gamma_{\alpha,0}$ lies on the Clifford torus that is the preimage of the equator of $S^2$ under the Hopf map; the segment of the curve on the same side of this torus as the observer is shown in black and the further segment is shown in gray. The preimage $\pi^{-1}(p)$ of a point on the equator of $S^2$ under the Hopf map is shown in red. Note the intersection of $\pi^{-1}(p)$ with $\gamma_{\alpha,0}([0,1])$ is the vertices of a square, a 3-design on $\pi^{-1}(p)$.}
\end{center}
\end{figure}

We can similarly use the construction of Theorem \ref{thm:stitch} to produce smooth, simple $t$-design curves on $S^3$ for $t\in\{4,5,6,7\}$ from the smooth, simple 2- and 3-design curves discovered by Ehler and Gr\"{o}chenig \cite[Example 3.3]{EhlerGrochenig23}. For $t\in\{0,...,27\}$, we can produce further explicit examples of $\lfloor t/2\rfloor$-design curves on $S^2$ by applying the construction of Ehler and Gr\"{o}chenig \cite[Section 5]{EhlerGrochenig23} which can be used to build a $\lfloor t/2\rfloor$-design curve on $S^2$ from a $\lfloor t/2\rfloor$-design on $S^2$ to the $\lfloor t/2\rfloor$-designs on $S^2$ compiled by Hardin and Sloane \cite[Table I]{HardinSloane02}, and the construction of Theorem \ref{thm:stitch} applied to these curves gives numerous additional examples of simple $t$-design curves on $S^3$.

We now discuss results related to Theorem \ref{thm:stitch} concerning $t$-design curves on spaces other than spheres. Define, for any $t\in\Bbb N$ and $n\in\Bbb N_+$, a $t$-design curve on a measure space $(\Sigma\subset\Bbb R^n,\mu)$ to be a continuous, piecewise smooth, closed curve $\gamma:[0,1]\to\Sigma$ with finitely many self-intersections whose associated line integral exactly averages the restrictions of elements of $P_t(\Bbb R^n)$ to $\Sigma$ as in Definition \ref{def:design}. Consider a $t$-design curve $\alpha$ on such $(\Sigma,\mu)$. The curve
\[[0,1]\ni s\to(\alpha((t+1)s-\left\lfloor(t+1)s\right\rfloor),e^{2\pi is})\in\Sigma\times S^1\]
can be shown exactly as in the proof of Theorem \ref{thm:stitch} (substituting for the Hopf map the projection $\Sigma\times S^1\to\Sigma$, which similarly gives rise to a principal $S^1$-bundle) to be a $t$-design curve on $(\Sigma\times S^1\subset\Bbb R^{n+2},\mu\times\sigma)$. As the curve 
\[[0,1]\ni s\mapsto e^{2\pi is}\in S^1\]
is trivially a $t$-design curve on $S^1$ for any $t\in\Bbb N$, we see that
\[\gamma_t:[0,1]\ni s\to(e^{2\pi i(t+1)s},e^{2\pi is})\in S^1\times S^1\]
is a $t$-design curve on the Clifford torus $S^1\times S^1$ of length
\[\ell(\gamma_t)=2\pi\sqrt{(t+1)^2+1}\leq2\pi\sqrt5 t.\]
We then see that for any $d\in\Bbb N_+$,
\begin{equation}\label{eq:gammatd}
\gamma_{t,d}:[0,1]\ni s\to(e^{2\pi i(t+1)^{d-1}s},e^{2\pi i(t+1)^{d-2}s},...,e^{2\pi is})\in (S^1)^d
\end{equation}
is a $t$-design curve on $(S^1)^d$ of length 
\[\ell(\gamma_{t,d})=2\pi\sqrt{\sum_{j=0}^{d-1}(t+1)^{2j}}\leq\mathscr C_{(S^1)^d}t^{d-1}\]
for the constant
\[\mathscr C_{(S^1)^d}:=2\pi\sqrt{\sum_{j=0}^{d-1}4^j}.\]

\begin{figure}
\begin{center}
\includegraphics[width=.8\textwidth]
{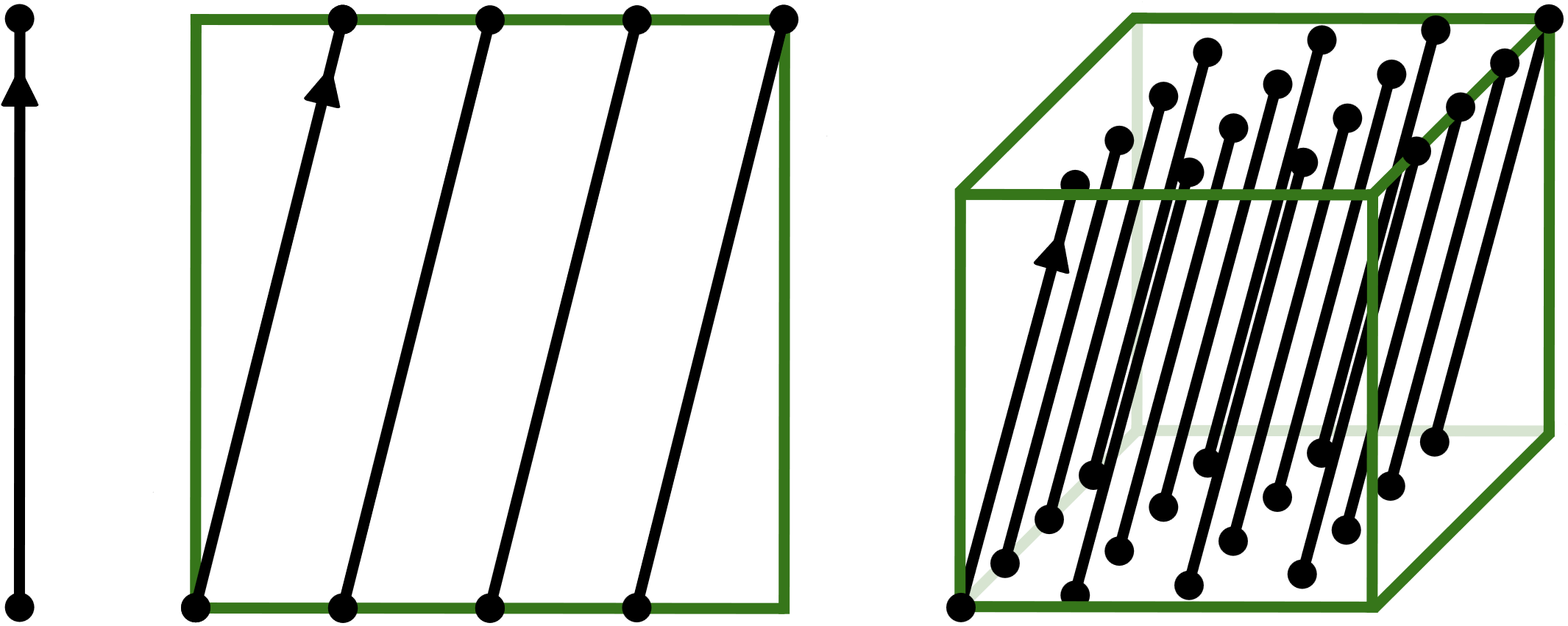}
\caption{\label{fig:Torusex} 3-design curves on $S^1$, $S^1\times S^1$, and $S^1\times S^1\times S^1$ constructed as in \eqref{eq:gammatd}. For a $\lfloor t/2\rfloor$-design curve $\alpha$ on $S^2$, the curve $\gamma_{\alpha,0}$ as in Theorem \ref{thm:stitch} can be constructed by mapping such a $t$-design curve on $S^1\times S^1$ onto the immersed torus $\pi^{-1}(\alpha([0,1]))$.}
\end{center}
\end{figure}

Noting that the curves \eqref{eq:gammatd} may be lengthened by any arbitrary amount as in the proof of Theorem \ref{thm:stitch} when $d>1$, for any such $d$, any $C\geq\mathscr C_{(S^1)^d}$, and $t\in\Bbb N_+$, we may construct a smooth, simple $t$-design curve on $(S^1)^d$ of length $Ct^{d-1}$.
We may similarly show using Theorem \ref{thm:EG2} when $n=2$ and Theorem \ref{thm:seq} when $n=3$ that there exists a constant $\mathscr C_{S^n\times (S^1)^d}$ such that for any $C\geq\mathscr C_{S^n\times (S^1)^d}$ and $t\in\Bbb N_+$, there exists a (simple, for $n=3$) $t$-design curve on $S^n\times(S^1)^d$ of length $Ct^{n-1+d}$. A proof analogous to one of Ehler and Gr\"{o}chenig \cite[Theorem 1.1]{EhlerGrochenig23} can be used to show that all of these curves achieve the minimal possible asymptotic order of length among such curves.

For $\Bbb K$ the reals or complex numbers, $k:=\dim_\Bbb R\Bbb K-1$, and $n\in\Bbb N_+$, we consider the $\Bbb K$-projective space 
\[\Bbb{KP}^n=\{[\omega]\:|\:\omega\in S^{(k+1)(n+1)-1}\subset\Bbb K^{n+1}\}\quad([\omega]:=\{\omega\zeta\:|\:\zeta\in S^k\subset\Bbb K\}),\]
which we realize as a subset of $\Bbb K^{(n+1)^2}\cong\Bbb R^{(k+1)(n+1)^2}$ via the embedding taking an element $[\omega]\in\Bbb{KP}^n$ to the matrix $\omega\omega^*$ which projects $\Bbb K^{n+1}$ to the span over $\Bbb R$ of $[\omega]$. This provides us with a notion of $t$-design curves on $\Bbb{KP}^n$. We may then use $t$-design curves on $S^2$ and $S^3$ with antipodally symmetric image (which can be built on $S^2$ using the construction described by Ehler and Gr\"{o}chenig \cite[Section 5]{EhlerGrochenig23} and on $S^3$ using the construction of Theorem \ref{thm:stitch}) to build $\left\lfloor t/2\right\rfloor$-design curves on $\Bbb{RP}^2$ and $\Bbb{RP}^3$ respectively with half the lengths of the original curves. For any $d\in\Bbb N$, we then get existence results for $t$-design curves analogous to those for $S^2\times(S^1)^d$ and $S^3\times(S^1)^d$ on $\Bbb{RP}^2\times (S^1)^d$ and $\Bbb{RP}^3\times (S^1)^d$. For any $n\in\Bbb N$, we also plan \cite{LindbladWIP} to provide an existence result for $t$-design curves on $\Bbb{CP}^n$ achieving an asymptotic order of length such that we can prove existence of $t$-design curves on $S^{2n+1}$ for $n>1$ asymptotically shorter than the current asymptotically shortest such curves shown to exist \cite[Theorem 1.3]{EhlerGrochenig23} by combining this existence result with a generalization of Theorem \ref{thm:stitch} which uses the complex projective map $S^{2n+1}\ni\omega\mapsto[\omega]\in\Bbb{CP}^n$ (which can be identified with the Hopf map when $n=1$ via the association $\Bbb{CP}^1\cong S^2$) to build a simple $t$-design curve on $S^{2n+1}$ of length approximately $(t+1)\ell(\alpha)$ from a $\lfloor t/2\rfloor$-design curve $\alpha$ on $\Bbb{CP}^n$.

\section*{Acknowledgements}
\label{sec:thanks}

The author would like to thank Henry Cohn for engaging in interesting conversations about problems related to the ones addressed by this manuscript. The author would also like to thank the School of Science, the Department of Mathematics, and the Office of Graduate Education for the MIT Dean of Science fellowship for support in their doctoral studies and would like to acknowledge NSF grant DMS-2105512 and the Simons Foundation Award \#994330 (Simons Collaboration on New Structures in Low-Dimensional Topology) for their support.

\bibliography{Curvesbib.bib}

\begin{thebibliography}{10}

\bibitem{BannaiBannai09}
Eiichi Bannai and Etsuko Bannai.
\newblock A survey on spherical designs and algebraic combinatorics on spheres.
\newblock {\em European J. Combin.}, 30(6):1392--1425, 2009.

\bibitem{Bondarenko...13}
Andriy Bondarenko, Danylo Radchenko, and Maryna Viazovska.
\newblock Optimal asymptotic bounds for spherical designs.
\newblock {\em Ann. of Math. (2)}, 178(2):443--452, 2013.

\bibitem{Bonnet1867}
Pierre~Ossian Bonnet.
\newblock M\'emoire sur la theorie des surfaces applicables sur une surface donn\'ee.
\newblock {\em J de l’Ecole Polytechnique}, 19(32):1--146, 1848.

\bibitem{Cohn...06}
Henry Cohn, John~H. Conway, Noam~D. Elkies, and Abhinav Kumar.
\newblock The {$D_4$} root system is not universally optimal.
\newblock {\em Experiment. Math.}, 16(3):313--320, 2007.

\bibitem{Delsarte...77}
Philippe Delsarte, Jean-Marie Goethals, and Johan~J. Seidel.
\newblock Spherical codes and designs.
\newblock {\em Geom. Dedicata}, 6(3):363--388, 1977.

\bibitem{EhlerGrochenig23}
Martin Ehler and Karlheinz Gr\"ochenig.
\newblock {$t$}-design curves and mobile sampling on the sphere.
\newblock {\em Forum Math. Sigma}, 11:Paper No. e105, 25, 2023.

\bibitem{Gauss1827}
Carl~Friedrich Gauss.
\newblock Disquisitiones generales circa superficies curvas.
\newblock {\em Commentationes societatis regi\ae scientiarum Gottingensis recentiores, Commentationes classis mathematic\ae}, 6:99--146, 1828.

\bibitem{HardinSloane02}
Ronald~H. Hardin and Neil J.~A. Sloane.
\newblock Mc{L}aren's improved snub cube and other new spherical designs in three dimensions.
\newblock {\em Discrete Comput. Geom.}, 15(4):429--441, 1996.

\bibitem{Koning98}
Hermann K\"{o}nig.
\newblock Cubature formulas on spheres.
\newblock In {\em Advances in multivariate approximation ({W}itten-{B}ommerholz, 1998)}, volume 107 of {\em Math. Res.}, pages 201--211. Wiley-VCH, Berlin, 1999.

\bibitem{Kuperberg06}
Greg Kuperberg.
\newblock Numerical cubature from {A}rchimedes' hat-box theorem.
\newblock {\em SIAM J. Numer. Anal.}, 44(3):908--935, 2006.

\bibitem{LindbladWIP}
Ayodeji Lindblad.
\newblock Lifting design curves.
\newblock In preparation.

\bibitem{Lindblad24}
Ayodeji Lindblad.
\newblock Designs related through projective and hopf maps.
\newblock Preprint, arXiv:2310.12091, 2023.

\bibitem{Okuda15}
Takayuki Okuda.
\newblock Relation between spherical designs through a {H}opf map.
\newblock Preprint, arXiv:1506.08414, 2015.

\end{thebibliography}

\end{document}